\documentclass[a4paper,11pt]{article}
 \usepackage[english]{babel}
 \usepackage{graphicx}
 \usepackage[latin1]{inputenc}
 \usepackage[T1]{fontenc}
 \usepackage{lmodern}
 \usepackage[normalem]{ulem}
 \usepackage{verbatim}
 \usepackage{bbm}
 \usepackage{ntheorem}
 \usepackage{stmaryrd}
 \usepackage{amsmath}
 \usepackage{amssymb}
\usepackage{hyperref}

\theoremheaderfont{\scshape}
\theoremseparator{.}
\newtheorem{theorem}{Theorem}[section]
\newtheorem{thm}[theorem]{Theorem}
\newtheorem{coro}[theorem]{Corollary}
\newtheorem*{lem}[theorem]{Lemma}
\newtheorem{quest}[theorem]{Question}
\newtheorem*{fuzzquest}[theorem]{Fuzzy question}
\newtheorem*{family}[theorem]{Family of questions}

\newtheorem{prop}[theorem]{Proposition}

\newenvironment{rem}[1][Remark.]{\begin{trivlist}
\item[\hskip \labelsep {\itshape #1}]}{\end{trivlist}}

\newenvironment{obs}[1][Observation.]{\begin{trivlist}
\item[\hskip \labelsep {\itshape #1}]}{\end{trivlist}}

\newenvironment{exam}[1][Example.]{\begin{trivlist}
\item[\hskip \labelsep {\textit{#1}}]}{\end{trivlist}}

\def\sqw{\hbox{\rlap{\leavevmode\raise.3ex\hbox{$\sqcap$}}$%
\sqcup$}}

\newcommand{\N}{\ensuremath{{\mathbb N}}}

\newcommand{\Z}{\ensuremath{\mathbb Z}}

\newcommand{\rpdeux}{\ensuremath{P}}

\newcommand{\defini}{\textbf}

\newcommand{\field}{\ensuremath{\mathbb{K}}}

\newcommand{\R}{\ensuremath{\mathbb R}}

\newcommand{\Q}{\ensuremath{\mathbb Q}}

\newcommand{\C}{\ensuremath{\mathbb C}}

\def\acts{\curvearrowright}

\newcommand{\Cpts}{\ensuremath{\mathsf{C}}}

\newcommand{\Stab}{\mathsf{Stab}}

\newcommand{\Aut}{\mathsf{Aut}}

\newcommand{\saut}{\vspace{0.25cm}}

\newenvironment{proof}{  
    \vspace*{-.4em}  {\it Proof.}%
}{
    \hfill\sqw\vspace*{.5em}
}

\newenvironment{proofnice}{  
    \vspace*{-.4em}  {\it Proof of Theorem~\ref{nice}.}%
}{
    \hfill\sqw\vspace*{.5em}
}

\author{S\'ebastien \sc{Martineau}\footnote{E-mail: sebastien.martineau@weizmann.ac.il.}\\
\\
{\it The Weizmann Institute of Science}}
\title{On rotarily transitive graphs}
\begin{document}
\maketitle

\begin{abstract}
From the point of view of discrete geometry, the class of locally finite transitive graphs is a wide and important one. The subclass of Cayley graphs is of particular interest, as testifies the development of geometric group theory. Recall that Cayley graphs can be defined as non-empty locally finite connected graphs endowed with a transitive group action
such that any non-identity element acts without fixed point.

We define a class of transitive graphs which are transitive in an ``absolutely non-Cayley way'': we consider graphs endowed with a transitive group action
such that any element of the group acts with a fixed point. We call such graphs \emph{rotarily transitive graphs}, and we show that, even though there is no finite rotarily transitive graph with at least 2 vertices, there is an infinite locally finite connected rotarily transitive graph. The proof is based on groups built by Ivanov which are finitely generated, of finite exponent and have a small number of conjugacy classes.

We also build infinite transitive graphs (which are not locally finite) any automorphism of which has a fixed point. This is done by considering ``unit distance graphs'' associated with the projective plane over suitable subfields of $\R$.
\end{abstract}

\section{Introduction}

This introduction splits into two parts. In a first part, we present the context, the general question concerning this paper and the answers we provide, in a way primarily meant to be straightforward. In the second part, we provide more detailed definitions and additional remarks, in order to make this paper as self-contained and comprehensive as possible.

\subsection{General introduction}

\label{straight}

A good framework for the study of discrete geometry is that of non-empty connected graphs where every vertex has only finitely many neighbours. Let us call such a graph a \defini{nice} graph. Among these, a wide and important class of graphs is that of transitive graphs\footnote{See e.g.\ \cite{benjaministflour}.}, which are the ones that look the same seen from any vertex. Formally, a graph is said to be \defini{transitive} (or \defini{homogeneous}) if its automorphism group acts transitively on its vertices.

Originating from the world of group theory, Cayley graphs form a very rich class of transitive graphs: a graph is a \defini{Cayley graph} if it is nice and can be endowed with a transitive group action such that any non-identity element acts without fixed point. (\footnote{Refer to \cite{delaharpe} for a presentation of the so-called geometric group theory and to \cite{sabid} for the equivalence between the ``generating system'' definition of a Cayley graph and the one presented here.})

In the opposite, one could be interested in graphs that can be endowed with a transitive group action such that any element of the group acts with at least one fixed point. Call such a graph a \defini{rotarily transitive graph}. One can also define a graph to be \defini{strongly rotarily transitive} if it is transitive and \emph{every} automorphism of the graph has at least one fixed point.
Studying rotarily transitive graphs amounts to grasping the ways there are for graphs to be transitive in a ``strongly non-Cayley way''.

\setcounter{theorem}{-1}

\begin{fuzzquest}
\label{qu:fuzzy}
Are there (strongly) rotarily transitive ``graphs''?
\end{fuzzquest}

Here are a few answers, depending on the meaning we give to ``graphs''.

\begin{prop}
\label{finitecase}
Any finite rotarily transitive graph has exactly 1 vertex.
\end{prop}

\begin{thm}
\label{nice}
There is a nice infinite rotarily transitive graph.
\end{thm}

\begin{thm}
There is an infinite strongly rotarily transitive connected graph.
\end{thm}

Proposition~\ref{finitecase} is essentially classical, and stems from the Cauchy-Fro\-be\-nius Lemma. The proof of Theorem~\ref{nice} has a group-theoretic flavour: it involves some infinite finitely generated torsion groups constructed by Ivanov. The graph is obtained by letting such a group act on one of its conjugacy classes and building a compatible graph structure. As for Theorem~\ref{ugly}, its proof is geometric: we construct a suitable graph by taking as vertex-set the elements of the projective plane over some subfield of $\R$ and connecting two vertices by an edge if they are at distance exactly some suitable $\ell$.

For a detailed version of the current subsection (Section~\ref{straight}) and a proof of the easy Proposition~\ref{finitecase}, see Section~\ref{detail}. Theorem~\ref{nice} is established in Section~\ref{algebra}, while Theorem~\ref{ugly} is proved in Section~\ref{geometry}. This paper ends with several questions in Section~\ref{sec:quest}.

\subsection{Notation, definition and results}

\label{detail}

In this paper, a \defini{graph} is a pair $(V,E)$ where $V$ is a set and $E$ a subset of ${E \choose 2}$, where ${E \choose 2}$ denotes the set of all subsets of $E$  having exactly two elements. A graph is often denoted by $\mathcal{G}=(V,E)$. The elements of $V$ are called the \defini{vertices} of $\mathcal{G}$ and the elements of $E$ are the \defini{edges} of $\mathcal{G}$.

Two vertices of a graph are called \defini{neighbours} if they are distinct and there is an edge of the graph containing them both. ``Being neighbours'' defines a symmetric relation on the set of vertices of the considered graph.
A graph is said to be \defini{empty} if it has no vertex, \defini{finite} if it has finitely many vertices, \defini{infinite} if it has infinitely many vertices, and \defini{locally finite} if every vertex has finitely many neighbours.

A \defini{path (of length $n$)} is a finite sequence of vertices $(v_0,\dots,v_n)$ such that for any $i<n$, $\{v_i,v_{i+1}\}$ is an edge. (\footnote{The integer $n$ may be equal to 0.}) Its \defini{extremities} are $v_0$ and $v_n$. A path is said to \defini{connect} its extremities. A graph is \defini{connected} if any two vertices are connected by a path. The \defini{diameter} of a connected graph is the smallest $k\in \N\cup\{\infty\}$ such that for all vertices $u$ and $v$, there is a path of length at most $n$ that connects $u$ and $v$.
We will say that a graph is \defini{nice} if it is non-empty, connected, and locally finite. Nice graphs are a natural framework for discrete geometry \cite{benjaministflour}.

\saut

Now that the geometric vocabulary has been defined, let us deal with the notion of symmetry. An \defini{automorphism} of a graph $\mathcal{G}=(V,E)$ is a bijection $\varphi$ from $V$ to $V$ such that
$$
\forall u,v \in V,~\{u,v\}\in E \iff \{\varphi(u),\varphi(v)\}\in E.
$$
The automorphisms of $\mathcal{G}$ form a group, which we denote by $\Aut(\mathcal{G})$.
When we will speak of a \defini{group $G$ acting on a \emph{graph} $\mathcal{G}$}, we will mean a morphism from $G$ to $\Aut(\mathcal{G})$, which can be seen as a left group action of $G$ on $V$ by graph automorphisms.

We will only work with left group actions. Recall that an \defini{orbit} of a group action $G \acts X$ is an equivalence class for the equivalence relation on $X$ defined as follows:
$$
x\sim y \iff \exists g \in G,~g\cdot x=y.
$$
A group action $G \acts X$ is \defini{transitive} if
$$
\forall x,y \in X,~\exists g\in G,~ g\cdot x = y,
$$
that is to say if it has at most one orbit.
Given a group action $G \acts X$, the \defini{stabiliser} of an element $x$ of $X$ is $\Stab(x):=\{g\in G: g\cdot x = x\}$, which is a subgroup of $G$, and $x$ is called a \defini{fixed point} of $g$ if $g \in \Stab(x)$.
A group action $G\acts X$ is \defini{faithful} if $\bigcap_{x\in X} \mathsf{Stab}(x)=\{1\}$.
A group action is \defini{free} if for any $g\in G\backslash\{1\}$, $g$ has no fixed point.
Finally, we say that a group acts \defini{by rotations} if for every $g\in G$, $g$ has at least one fixed point.

\begin{rem}
As soon as $G$ is \defini{non-trivial}, i.e.\ as soon as it does not consist of a single element, no action of $G$ by rotations can be free. Actually, one can think ``being by rotations'' as a strong negation of the notion of freeness for actions of non-trivial groups.

Also notice that if a group $G$ acts by rotations on a set $X$, then the set $X$ cannot be empty: indeed, the identity element of $G$ must have a fixed point.
\end{rem}

We define a group action to be \defini{rotarily transitive} if it is transitive and by rotations.

\begin{exam}
The usual action of $\mathsf{SO}(3)$ on $S_2 := \{x\in \R^3:x_1^2+x_2^2+x_3^2=1\}$ is rotarily transitive.
\end{exam}

\begin{exam}
The usual action of $\mathsf{GL}(3,\R)$ on the set of the lines of $\R^3$ which contain 0 is rotarily transitive.
\end{exam}

\begin{exam}
\label{finitesupp}
Let $X$ be an infinite set. Let $G$ denote the group formed of the bijections $\sigma : X \to X$ of finite support. (\footnote{The support of $\sigma$ is $\{x\in X : \sigma(x)\not=x\}$.}) The obvious action of $G$ on $X$ is rotarily transitive.
\end{exam}

\saut

A graph $\mathcal{G}=(V,E)$ is said to be \defini{transitive} (or \defini{homogeneous}) if it can be endowed with a transitive group action, which is the same as asking for the usual action of $\Aut(\mathcal{G})$ on $V$ to be transitive. A graph is a \defini{Cayley graph} if it is nice and can be endowed with a free transitive group action. \begin{small}See \cite{sabid} for the equivalence between this definition and another common one.\end{small}

\begin{rem}
A Cayley graph \emph{may} have non-trivial automorphisms with a fixed point, as testifies the bi-infinite line. The group $\mathbb{Z}$ acts freely and transitively on it by translations, and flipping the line around some vertex defines a graph automorphism with a fixed point. Of course, such an automorphism cannot be ``used'' in any free transitive group action on this graph.
\end{rem}

We define a graph to be \defini{rotarily transitive} (or \defini{rotarily homogeneous}) if it can be endowed with a rotarily transitive group action. We say that a graph is \defini{strongly rotarily transitive} if the usual action of $\Aut(\mathcal{G})$ on $V$ is rotarily transitive --- which is the same as asking for the graph to be transitive and for every action on it to be by rotations.

\begin{rem}
Every strongly rotarily transitive graph is rotarily transitive, and every rotarily transitive graph is non-empty.
\end{rem}

\begin{exam}
We say that a \emph{graph} is \defini{trivial} if its set of vertices contains exactly one element. Trivial graphs obviously exist, and any trivial graph is strongly rotarily transitive and Cayley.
\end{exam}

\begin{small}
\begin{rem}
Conversely, if a graph $\mathcal{G}=(V,E)$ is strongly rotarily transitive and Cayley, then it is trivial. Indeed, if $G$ acts rotarily transitively and freely on $V$, then any element of $G\backslash\{1\}$ should have 0 and at least 1 fixed point, so that $G$ consists of a single element. If $G$ consists of a single element and acts rotarily transitively on a set, then this set must contain exactly one point.
\end{rem}
\end{small}

\setcounter{theorem}{-1}
\begin{family}
For $Y\in\{$``non-trivial finite'', ``nice non-trivial'', ``non-trivial''$\}$ and $Z\in \{\text{``''},\text{``strongly''}\}$, one can ask the following question.

\begin{center}
Is there a $Y$ $Z$ rotarily transitive graph?
\end{center}
\end{family}

We answer hereafter this question for all values of $(Y,Z)$ but (``nice non-trivial'', ``strongly'') --- which leads us to ask Question~\ref{questrong} in Section~\ref{sec:quest}.

\begin{prop}
There is no non-trivial finite rotarily transitive graph. \begin{small}In particular, there is no non-trivial finite strongly rotarily transitive graph.\end{small}
\end{prop}

\begin{thm}
There is a nice non-trivial rotarily transitive graph.
\end{thm}

\begin{thm}
\label{ugly}
There is a non-trivial strongly rotarily transitive graph.
\begin{small}Actually, for every integer $k\geq 2$, there is an infinite strongly rotarily transitive connected graph with countably many vertices and diameter $k$.\end{small}
\end{thm}

We end the current subsection with the proof of Proposition~\ref{finitecase} and a few comments. Section~\ref{algebra} is devoted to the proof of Theorem~\ref{nice}. The construction is algebraic and elementary, once one takes for granted a hard theorem due to Ivanov (Theorem~\ref{thm:ivanov}). Theorem~\ref{ugly} is established in Section~\ref{geometry}. The construction is geometric and elementary. We conclude this paper with several questions in Section~\ref{sec:quest}.

\subsubsection*{Proof of Proposition~\ref{finitecase}}

Recall the following easy combinatorial lemma. See e.g.~\cite{cauchyfrob}.

\begin{lem}[Cauchy-Frobenius]
Let $G$ be a finite group acting on a finite set $X$. For $g\in G$, denote by $F_g$ the number of fixed points of $g$. Then, the number of orbits of the considered action is equal to $\frac{1}{|G|}\sum_{g\in G} F_g$.
\end{lem}

Proposition~\ref{finitecase} is a particular case\footnote{Indeed, if there is a group acting rotarily transitively on a finite graph $\mathcal{G}=(V,E)$, then there is such a finite group (which can be taken to be a subgroup of $\mathfrak{S}_V$).} of the following easy result, which is well-known \cite{jordan, serrejordan} --- even though usually phrased differently.

\setcounter{theorem}{3}
\begin{coro}[Jordan]
\label{finitenograph}
Let $G$ be a finite group acting rotarily transitively on a (necessarily finite) set $X$. Then, $X$ is a singleton.
\end{coro}

\begin{proof}
The set $X$ is necessarily finite because the group $G$ is finite and acts transitively on it. Since $G$ acts by rotations on it, the set $X$ is not empty.
By the Cauchy-Frobenius Lemma, because the action under study is transitive and because $X\not= \varnothing$, we have, keeping the notation of the lemma,
$$
\frac{1}{|G|}\sum_{g\in G} F_g=1.
$$
Since the considered action is by rotations, we have $\forall g\in G,~F_g \geq 1$. These two observations imply that $\forall g\in G,~F_g = 1$. In particular, the identity element of $G$ has a unique fixed point, meaning that $X$ contains a unique element.
\end{proof}

\subsubsection*{Additional remarks}

The connectedness assumption appearing in the definition of ``nice'' is crucial as far as (non strong) rotary transitivity is concerned. Indeed, it is easy to build an infinite locally finite rotarily transitive graph, which suffices to answer the question associated with $(Y,Z)=(\text{``non-trivial''},\text{``''})$. Take $\mathcal{G}:=(X,\varnothing)$, where $X$ is a set endowed with a rotarily transitive action $G\acts X$. This group action is by graph automorphisms, since $\mathcal{G}$ has no edge. To build a graph as desired, it is thus enough to know rotarily transitive group actions on infinite sets (not necessarily on graphs). The usual action of $\mathsf{SO}(3)$ on $S_2$ or the ``finite support example'' of page~\pageref{finitesupp} do the trick. (\footnote{Likewise, the graph $\left( X, {X \choose 2}\right)$ is an easy example of an infinite connected rotarily transitive graph. It is not locally finite.}) Theorem~\ref{nice} is not as shallow.

As for strong rotary transitivity, the situation is different: assuming the Axiom of Choice, we show that every strongly rotarily transitive graph is connected. To do so, take $\mathcal{G}=(V,E)$ to be a strongly rotarily transitive graph. Since $\mathcal{G}$ is transitive, its connected components\footnote{A \defini{connected component} $V_0$ of $\mathcal{G}=(V,E)$ is an equivalence class for the equivalence relation ``being connected by a path''. It is endowed with a graph structure by taking the edge set to be $E\cap {V_0 \choose 2}$.} are pairwise isomorphic. Denote by $\mathcal{G}_0 =(V_0,E_0)$ the connected component of some vertex $v_0$ of $\mathcal{G}$. Such a component exists because the graph $\mathcal{G}$ is rotarily transitive hence non-empty. Let $\Cpts$ denote the set of the connected components of $\mathcal{G}$. Let $\mathcal{H}$ denote the graph with vertex-set $\Cpts \times V_0$ and such that there is an edge between $(c,v)$ and $(c',v')$ if and only if $c=c'$ and $\{v,v'\} \in E_0$. The Axiom of Choice implies that $\mathcal{G}$ and $\mathcal{H}$ are isomorphic. Now, assume for contradiction that $\mathcal{G}$ is not connected. Then, $\Cpts$ contains at least 2 elements. By using the Axiom of Choice, we can get\footnote{A \defini{pair} is a set containing exactly 2 elements. By Zorn's Lemma, one can find a set $S$ of disjoint pairs of elements of $\Cpts$ that is maximal with this property. This means that at most one \emph{element} of $\Cpts$ does not belong to an element of $S$. If there is no such \emph{element}, define $\sigma$ by mapping an element of $\Cpts$ to the other element of the pair in $S$ containing it. Otherwise, since $\Cpts$ is not a singleton, one can choose a pair belonging to $S$ and define $\sigma$ as before on the elements of the other pairs belonging to $S$ and by cyclically permuting the remaining three elements of $\Cpts$.} a bijection $\sigma : \Cpts \to \Cpts$ without any fixed point. The map $(c,v)\mapsto (\sigma(c),v)$ is a graph automorphism of $\mathcal{H}$, and it has no fixed point. Thus $\mathcal{H}$, hence $\mathcal{G}$, cannot be strongly rotarily transitive: the claim is established. 

\saut

Another remark is that a non-trivial bipartite connected graph cannot be rotarily transitive. \begin{small}Recall that a graph is \defini{bipartite} if there is a way to colour its vertices black or white so that there is no monochromatic edge. More formally, a graph $\mathcal{G}=(V,E)$ is bipartite if there is a function $f:V\to \{0,1\}$ such that $\forall e\in E,~f(e)=\{0,1\}$. The assertion is proved as follows. Let $\mathcal{G}=(V,E)$ be a non-trivial bipartite connected graph. Assume that $\mathcal{G}=(V,E)$ is rotarily transitive, hence non-empty. Let $f:V\to \{0,1\}$ be such that $\forall e\in E,~f(e)=\{0,1\}$. Since $\mathcal{G}$ is connected and has at least two vertices, it must have at least one edge. Denote by $\{v_0,v_1\}$ some edge of $\mathcal{G}$.
Let $G\acts \mathcal{G}$ be a rotarily transitive group action, and let $g\in G$ be such that $g\cdot v_0=v_1$. Since $\mathcal{G}$ is connected and $f(v_0)=1-f(v_1)$, for every $v\in V$, we have $f(g\cdot v)=1-f(v)$. As a result, $g$ has no fixed point, contradicting the rotary transitivity of $G\acts \mathcal{G}$ and establishing the assertion.

In particular, non-trivial trees cannot be rotarily transitive. Recall that a \defini{tree} is a connected graph $\mathcal{G}=(V,E)$ such that for every $n\geq 3$, there is no injective map $\phi$ from $Z/n\Z$ to $V$ such that for every $i\in \Z/n\Z$, $\phi(i)$ and $\phi(i+1)$ are neighbours.\end{small}

\section{Group theory gives a nice weak example of a rotarily transitive graph}

\label{algebra}

In order to establish Theorem~\ref{nice}, we will need Theorem~\ref{thm:ivanov}. For the sake of completeness, let us a recall the relevant terminology.

Recall that a group $G$ acts on itself by \defini{conjugation} via $g\cdot h = ghg^{-1}$. In this section, this ``dot notation'' will always refer to that precise action (possibly restricted to a smaller set than the whole group). The orbits of $G$ for this action are called the \defini{conjugacy classes}, and the stabiliser of an element $g\in G$ is called its \defini{centraliser}. \begin{small}Notice that for $g,h\in G$, the element $g$ belongs to the centraliser of $h$ if and only if $gh=hg$, which in turn is equivalent to $h$ belonging to the centraliser of $g$.\end{small}

Finally, a \defini{proper} subgroup $G$ is a subgroup of $G$ which is neither \emph{equal} to $\{1\}$ nor to $G$.

\begin{thm}[Ivanov]
\label{thm:ivanov}
For any prime $p$ large enough, there is a group $G$ which satisfies the following conditions:
\begin{enumerate}
\item every proper subgroup of $G$ is isomorphic to the cyclic group $\Z/p\Z$, \label{monster}
\item $G$ has exactly $p$ conjugacy classes, \label{nbconj}
\item and $G$ is not isomorphic $\Z/p\Z$. \label{nontriv}
\end{enumerate}
\end{thm}

For this theorem, refer to \cite{olshdefining}, and in particular to the last two pages of the book.

Let $G$ be a group satisfying the conditions of Theorem~\ref{thm:ivanov}. By Condition~\ref{nbconj} and because $p$ is prime, if $G$ is abelian, then it is isomorphic to $\Z/p\Z$. Thus, by Condition~\ref{nontriv}, $G$ cannot be abelian. It is thus possible to take $g$ in $G\backslash\{1\}$, and $h$ to belong to $G$ but not to the subgroup generated by $g$. The subgroup of $G$ generated by $g$ and $h$ admits a proper subgroup: since $\Z/p\Z$ does not, by Condition~\ref{monster}, the group generated by $g$ and $h$ needs to be $G$.

By Condition~\ref{monster} \emph{and} because $G$ is not abelian, every element of $G\backslash \{1\}$ has order $p$.
It is also well-known that $G$ is necessarily infinite. This can be seen as a consequence of Corollary~\ref{finitenograph} \emph{and} the \emph{fact} that for any $g\in G\backslash \{1\}$, the action of $G$ by conjugation on the conjugacy class of $g$ --- which contains at least 2 elements --- is rotarily transitive. See the proof of Theorem~\ref{nice} for a full justification of the \emph{fact}.
Therefore, any group as in Theorem~\ref{thm:ivanov} is a solution to the Burnside Problem. (\footnote{Say that a group $G$ has \defini{finite exponent} if there is a positive integer $n$ such that $\forall g\in G,~g^n=1$. Recall that a group is \defini{finitely generated} if it admits a finite subset generating it as a group. The \defini{Burnside Problem} is the following question: does there exist an infinite finitely generated group of finite exponent? This question was answered in the affirmative in \cite{anov1, anov2, anov3}.}) \begin{small}
\end{small}

Given a group $G_0$ such that $\forall g\in G_0,~g^p=1$, it is also known that for any $g\in G_0$ and $m,n\in \Z$, if $g^m$ and $g^n$ belong to the same conjugacy class, then they are equal. To prove this, take $g,h\in G_0$ and $m,n\in \Z$ such that $g^m=h\cdot g^n$ and $g^m\not=g^n$. As $h\cdot 1=1$, we have $g^n\not=1$. The map $\sigma : g'\mapsto h\cdot g'$ is well-defined from $\langle g \rangle$ --- the subgroup of $G_0$ generated by $g$ --- to itself because, as $g$ is of prime order $p$, the element $g^n\not=1$ generates $\langle g\rangle$. Since $\sigma$ is not the identity map and because $h$ has order $p$, this map is a permutation of $\langle g \rangle$ of order $p$. As $\langle g \rangle$ has cardinality $p$, the permutation $\sigma$ must be a $p$-cycle without any fixed point. This contradicts $\sigma(1)=1$.

As a result, any non-abelian group $G_0$ satisfying Condition~\ref{monster} has at least $p$ conjugacy classes: indeed, if $G_0$ is such a group, then $G_0$ is non-trivial and every element of $G_0\backslash\{1\}$ has order $p$. 
Likewise, if a group $G$ satisfies the conditions of Theorem~\ref{thm:ivanov}, then every proper subgroup of $G$ intersects every conjugacy class exactly once.

\saut

\begin{proofnice}
Let $G$ be a group satisfying the conditions of Theorem~\ref{thm:ivanov}. We have seen that $G$ can be generated by two distinct elements: let us denote such a generating pair by $(g_1,g_3)$, and set $g_2 := g_3\cdot g_1$.

Let $V$ be the conjugacy class of $g_1$. Since $V$ is an orbit of $G$ for the action by conjugation, $G$ \emph{acts on} $V$ by conjugation, and this action is \emph{transitive}.
Define two elements $h$ and $h'$ of $V$ to be connected by an edge if there is some $g\in G$ such that $g\cdot h=g_i$ and $g\cdot h'=g_j$, with $\{i,j\}=\{1,2\}$.
The graph under consideration --- $\mathcal{G}$ --- is \emph{neither empty nor trivial}, i.e.\ $V$ contains at least two elements. Indeed, $V$ contains $g_1$ and if $V$ was a singleton, then $g_1$ would commute with $g_3$, implying that $G$ is abelian. But, we have seen that $G$ cannot be abelian.

Notice that $G$ acts on $V$ \emph{by graph-automorphisms}.
\begin{small}
Since the group $G$ acts on the set $V$, it is enough to show that for any $(h,h')\in V^2$ and $g\in G$, if $h$ and $h'$ are connected by an edge, then $g\cdot h$ and $g\cdot h'$ are connected by an edge. Let thus $(h,h')\in V^2$ be such that $h$ and $h'$ are connected by an edge, and let $g\in G$. Since $h$ and $h'$ are connected by an edge, we can take $g'$ such that $g'\cdot h=g_i$ and $g'\cdot h'=g_j$, with $\{i,j\}=\{1,2\}$. Take $g''$ to be $g'g^{-1}$. We have $g''\cdot(g\cdot h)=g''g\cdot h=g'g^{-1}g\cdot h = g'\cdot h = g_i$. Likewise, $g'' \cdot (g\cdot h')=g_j$. Hence, $g\cdot h$ and $g\cdot h'$ are connected by an edge.
\end{small}

In order to show that the group action $G\acts V$ is \emph{rotarily transitive}, it remains to show that every element of $g\in G$ has at least one fixed point in $V$. Let $g\in G$. If $g=1$, then $g$ has a fixed point, since $V$ is non-empty: let us thus assume that $g\not= 1$. We have seen that, necessarily, $g$ generates a group intersecting every conjugacy class exactly once. Thus, there is some integer $k$ such that $g^k\in V$. Fixing such a $k$, we have $g\cdot g^k=gg^kg^{-1}=g^{1+k-1}=g^k$. As a result, $g$ has a fixed point in $V$.

Let us prove that the considered graph is \emph{connected}. Note that for any $g\in G$ and $k\in \Z$, there is an edge between $g\cdot g_1$ and $gg_1^kg_3\cdot g_1$ --- indeed, $g_1^{-k}g^{-1}\cdot(g\cdot g_1)=g_1^{-k}\cdot g_1=g_1$ and $g_1^{-k}g^{-1}\cdot(gg_1^kg_3\cdot g_1)=g_3 \cdot g_1=g_2$. From this, it results (by induction) that $g_1$ is connected by a path to any $(g_1^{k_1}g_3g_1^{k_2}g_3\dots g_1^{k_n}g_3)\cdot g_1$. Because the $k_i$'s can be taken equal to 0 and $g_3$ has finite order, $g_1$ can be connected by a path to any $g\cdot g_1$, where $g$ belongs to the subgroup generated by $g_1$ and $g_3$. Since $\{g_1,g_3\}$ generates $G$ and $G$ acts transitively on $V$, the graph is connected.

It remains to show that the considered graph is \emph{locally finite}. 
First, notice that the centraliser of $g_1$ is finite. Indeed, were it not the case, the centraliser of $g_1$ would be infinite, hence $G$ by Condition~\ref{monster}. As a result, $V$ would be equal to $\{g_1\}$. But we have seen that $V$ contains at least two elements: the centraliser of $g_1$ must thus be finite. Since $g\cdot g_2=g_2 \iff gg_3\cdot g_1 =g_3\cdot g_1\iff g_3^{-1}gg_3\cdot g_1=g_1$, the centraliser $\mathsf{Cent}(g_2)$ of $g_2$ is equal to $g_3\mathsf{Cent}(g_1)g_3^{-1}$. Consequently, the centraliser of $g_2$ is also finite.

Because $G$ acts transitively on $\mathcal{G}$, establishing the local finiteness of $\mathcal{G}$ is reduced to proving that
\begin{enumerate}
\item there are only finitely many $h\in V$ such that there is some $g\in G$ such that $g \cdot g_1=g_1$ and $g \cdot h=g_2$
\item and there are only finitely many $h\in V$ such that there is some $g\in G$ such that $g\cdot g_2=g_2$ and $g \cdot h=g_1$.
\end{enumerate}
But these points result from the finiteness of the centralisers of $g_1$ and $g_2$.
Theorem~\ref{nice} follows.
\end{proofnice}

\begin{rem}
Actually, the construction presented in this proof shows that there is a nice infinite graph that can be endowed with a transitive group action $G\acts \mathcal{G}$ such that
every element of $G\backslash \{1\}$ has exactly one fixed point\footnote{Indeed, we have seen that the centraliser of $g_1$ is finite. Since we could have taken $g_1$ to be any element of $G\backslash\{1\}$ --- and $g_3$ to be any element of $G\backslash \langle g_1\rangle$ ---, every element of $G\backslash\{1\}$ has a finite centraliser. For $g\in G\backslash\{1\}$, the centraliser of $g$ is thus a proper subgroup of $G$, as it is finite and contains $g$. We have seen that such a subgroup of $G$ must intersect $V$ exactly once.}
and every vertex has a finite stabiliser for the group action $G\acts \mathcal{G}$.
\end{rem}

\section{Geometry provides wild examples of strongly rotarily transitive graphs}

\label{geometry}

Given a metric space $(X,d)$ and a real number $\ell$, we say that a mapping $\varphi : X \to X$ is an \defini{$\ell$-isometry} if it is bijective and satisfies
$$
\forall x,y\in X,~d(x,y)=\ell \iff d(\varphi(x),\varphi(y))=\ell.
$$
An \defini{isometry} of $(X,d)$ is a bijection from $X$ to $X$ that is an $\ell$-isometry for every $\ell$.

Beckman and Quarles proved in \cite{beckmanquarles} that for any $\ell > 0$, every $\ell$-isometry of the Euclidean plane is necessarily an isometry. A similar result holds for the sphere: see Theorem~\ref{thmeverling}.

Let $N\geq 2$. Let $S_N:= \left\{x\in \R^{N+1} : \sum_{i=1}^{N+1} x_i^2=1\right\}$. For any two elements $x$ and $y$ of $S_N$, define $d_N(x,y) \in [0,\pi/2]$ by
$\cos d_N(x,y) =  \langle x | y \rangle$, where $\langle \phantom{x} | \phantom{x} \rangle$ denotes the usual scalar product on $\R^N$. The function $d_N$ is well and univocally defined, and $(S_N,d_N)$ is a metric space. \begin{small}See 2.2 and 2.3~(1) in \cite{bridsonhaefliger}.\end{small}

\begin{thm}[Everling, \cite{isomsphere}]
\label{thmeverling}
For every $N\geq 2$ and $\ell \in (0,\pi/2)$, every $\ell$-isometry of $(S_N,d_N)$ is an isometry.
\end{thm}

We also recall the following well-known proposition.

\begin{prop}
\label{proportho}
Every isometry of $(S_N,d_N)$ is induced by an element of the orthogonal group $\mathsf{O}(N+1)$.
\end{prop}

\begin{small}For a proof, see 8.1.5, 9.7.1 and \underline{18.5.2} in \cite{bergeom1, bergeom2} or 4.3, 4.12 (1) and \underline{4.14} in \cite{bridsonhaefliger}.\end{small}

\vspace{0.2cm}

Combining Theorem~\ref{thmeverling} and Proposition~\ref{proportho} almost gives us an infinite strongly rotarily transitive graph. Indeed, for any metric space $(X,d)$ and real number $\ell$, one can define a graph $\mathcal{G}(X,d,\ell)$ by taking $X$ to be its vertex-set and connecting $x$ and $y$ by an edge if and only if $d(x,y)=\ell$. For every $\ell\in (0,\pi/2)$, the graph $\mathcal{G}(S_2,d_2,\ell)$ is infinite, transitive, and it admits exactly one automorphism without fixed point, namely $x\mapsto -x$, as given by Theorem~\ref{thmeverling}, Proposition~\ref{proportho} and the classification of the elements of $\mathsf{O}(3)$.

\vspace{0.2 cm}

In order to build infinite strongly rotarily transitive connected graphs, we adapt this strategy. Given a subfield $\field$ of $\R$, let $P(\field)$ denote the projective plane over  $\field$, i.e.\ the set consisting of the 1-dimensional $\field$-linear-subspaces of $\field^2$. It is equipped with the following distance:
$$
d(p,q):=\min\{d_2(x,y):x\in \R p\cap S_2,~ y \in \R q\cap S_2\}.
$$
When $\field$ is equal to $\R$, this metric space is called the \defini{elliptic plane}.

\begin{prop}
\label{projbq}
Let $\field$ denote a subfield of $\R$ which is closed under taking square roots of positive elements. Let $\ell \in (0,\pi/4)$ be such that $\cos(\ell)\in\field$ and there is an equilateral triangle of side length $\ell$ in $S_2$ with angle $\alpha \notin \pi \Q$. Then, every $\ell$-isometry of $P(\field)$ is an isometry. Besides, $\mathcal{G}(P(\field),d,\ell)$ is connected and its diameter is $\lceil \frac{\pi}{2\ell}\rceil$.
\end{prop}

\begin{proof}
Let $(\field, \ell, \alpha)$ satisfy the assumptions of Proposition~\ref{projbq}.
Let $\rpdeux$ denote $P(\field)$. Let $\varphi : \rpdeux \to \rpdeux$ be an $\ell$-isometry. We will show that $\varphi$ is an isometry, that $\mathcal{G}(\rpdeux,d,\ell)$ is connected and that its diameter is $\lceil \frac{\pi}{2\ell}\rceil$.

Let $r$ denote a rotation of angle $\alpha$ in $\mathsf{SO}(3)$. Let $x$ belong to its axis and let $y\in S_2$ be at distance $\ell$ from $x$. For $n\geq 0$, define $\ell_n$ to be $d(y,r^n(y))$. Note that this number does not depend on the choice of $(r,x,y)$, as long as it is legit. We claim that for every $n$, the bijection $\varphi$ is an $\ell_n$-isometry.

\begin{figure}[h!!]
\centering
\includegraphics[width=8 cm]{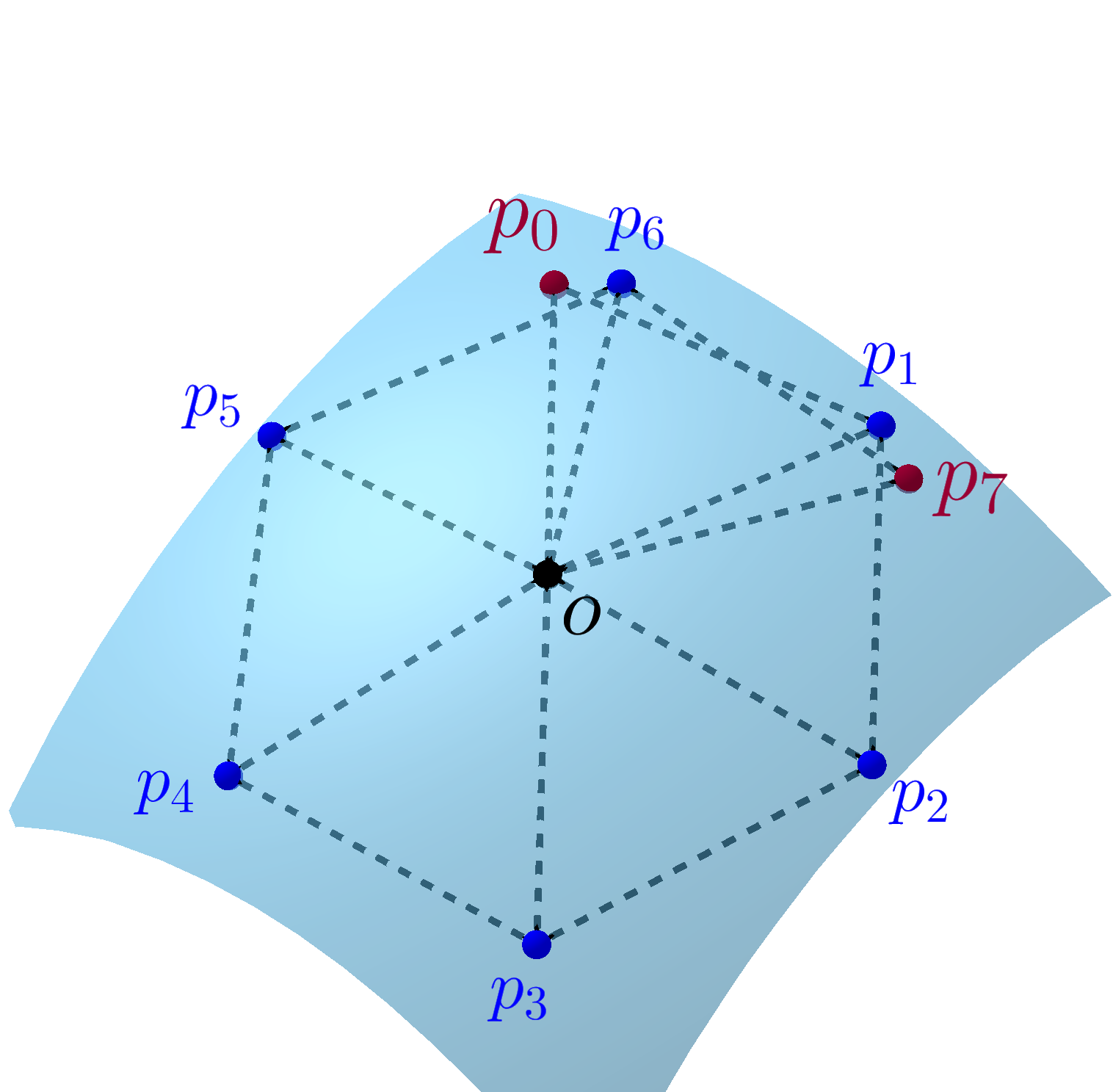}

\vspace{0.36 cm}

{This figure illustrates the characterisation of being\\at distance $\ell_n$. On the picture, $n$ is equal to 7.}

\vspace{0.2 cm}
\end{figure}

To prove this, it suffices to establish the following characterization. Two points $p$ and $q$ of $\rpdeux$ are at distance $\ell_n$ if and only if there are a point $o\in \rpdeux$ and points $p_0,\dots,p_n \in \rpdeux$ such that the following conditions hold:
\begin{itemize}
\item $p_0=p$,
\item $p_n=q$,
\item $\forall i \leq n,~d(p_i,o)=\ell$,
\item $\forall i <n,~d(p_i,p_{i+1})=\ell$,
\item $\forall i < n-1,~ p_i \not=p_{i+2}$.
\end{itemize}
If $\field=\R$, this is easy: the corresponding statement with $S_2$ instead of $\rpdeux$ holds and, because $\ell<\pi/4$, the projection map from $S_2$ to $\rpdeux$ induces an isometry from any closed ball of radius $\ell$ to its image. As a result, even without assuming that $\field=\R$, if a suitable $(o,p_0,\dots,p_n)$ exists, then $d(p,q)$ is equal to $\ell_n$.  The converse results from the following observation.

\begin{obs}Let $\ell'\in (0,\pi/2)$ be such that $\cos(\ell')\in \field$. Let $(p',q')\in \rpdeux^2$ be such that $d(p',q')\leq 2\ell'$. Let $x\in \R p'\cap S_2$ and $y \in \R q'\cap S_2$ be such that $d(p',q')=d_2(x,y)$. Let $z$ be a vector of $\field^3\backslash \{0\}$ which is orthogonal to $x$ and $y$ (the Gram-Schmidt process provides such vectors). Then, there is some positive $\lambda\in \field$ such that $r':= \field(x+y+\lambda z)$ satisfies $d(p',r')=d(r',q')=\ell'$.

\begin{small}Indeed, $d(p',r')=d(r',q')$ always holds, $\cos(\ell')$ belongs to $\field$, the field $\field$ is closed under taking square roots of positive elements, and $d(p',r')=\ell'$ is equivalent to the following polynomial equation of degree 2 with coefficients in $\field$: $$\cos(\ell')^2\|x+y+\lambda z\|_2^2=\langle x,x+y\rangle^2.$$
Besides, since $d_2(x,y)\leq 2\ell'$, taking $\lambda$ close enough to 0 or $\infty$ allows us to observe both inequalities in the equation above. Also notice that as $\field$ is closed under taking square roots of positive elements, $x$ and $y$ belong to $\field^3$.\end{small}
\end{obs}

As we have proved the observation, the characterisation of being at distance $\ell_n$ is established.

\saut

Given a subset $A$ of $\R$, say that $\varphi$ is an $A$-isometry if
$$
\forall x,y \in \rpdeux,~d(p,q)\in A \iff d(\varphi(p),\varphi(q))\in A.
$$
Let $\ell' \in (0,\pi/2)$ be such that $\cos(\ell')\in \field$. By the observation, we have $$\forall p,q\in \rpdeux,~d(p,q)\leq 2\ell' \iff \exists o\in \rpdeux,~d(o,p)=d(o,q)=\ell'.$$
Moreover, for any $k\geq 2$ and $p,q \in \rpdeux$, we have $d(p,q)\leq k\ell'$ if and only if there are $p_0,\dots,p_k$ in $\rpdeux$ such that:
\begin{itemize}
\item $p_0=p$,
\item $p_k=q$,
\item $\forall i <k-1,~d(p_i,p_{i+1})\in\{0,\ell'\}$,
\item $d(p_{k-1},p_k)\leq 2\ell'$.
\end{itemize}
As a result, if $\varphi$ is an $\ell'$-isometry, then it is a $[0,k\ell']$-isometry for every $k\geq 2$. Furthermore, the graph $\mathcal{G}(P,d,\ell')$ is connected and has diameter $\min (2, \lceil\frac{pi}{2\ell'}\rceil)$.

Let $p$ and $q$ belong to $P$. Let $\varepsilon >0$ and let us show that
$$|d(p,q)-d(\varphi(p),\varphi(q))| < \epsilon.$$
Since this is clear if $p=q$, we assume that $d(p,q)>0$.
As $\alpha$ does not belong to $\pi \Q$, we can pick an $n$ such that $\ell_n \in (0,\epsilon)$ and $2\ell_n <d(p,q)$. Let $k$ be the integer such that $k\ell_n < d(p,q)\leq (k+1)\ell_n$. Notice that $k\geq 2$ and $\cos(\ell_n)\in \field$. Since $\varphi$ is a $[0,k\ell_n]$- and a $[0,(k+1)\ell_n]$-isometry, we have $k\ell_n < d(\varphi(p),\varphi(q))\leq (k+1)\ell_n$. As a result, $|d(p,q)-d(\varphi(p),\varphi(q))| < \ell_n <\varepsilon$. The bijection $\varphi$ is thus an isometry.
\end{proof}

Recall the following proposition.

\begin{prop}
\label{proproj}
Let $\field$ be a subfield of $\R$ that is closed under taking square roots of positive elements. Every isometry of $P(\field)$ is induced by an element of $\mathsf{SO}(3,\R)\cap\mathsf{GL}(3,\field)$.
\end{prop}

For an elementary proof of this result when $\field=\R$, see \cite{projrefmo} or 9.7.1 and 19.1.2.2 in \cite{bergeom1, bergeom2}. The proof of \cite{projrefmo} holds for any subfield of $\R$ that is closed under taking square roots of positive elements.

\setcounter{section}{1}
\setcounter{theorem}{2}

\begin{thm}
\label{county}
For every integer $k\geq 2$, there is an infinite strongly rotarily transitive connected graph with countably many vertices and diameter $k$.
\end{thm}

\setcounter{section}{3}

\begin{proof}
For $k\geq 3$, one can find $\ell\in (0,\pi/4)$ such that $\lceil \frac{\pi}{2\ell}\rceil =k$ and there is an equilateral triangle of side length $\ell$ in $S_2$ with angle $\alpha \notin \pi \Q$. Fix such an $\ell$ and set $\field$ to be the smallest subfield of $\R$ that contains $\cos(\ell)$ and is closed under taking square roots of positive elements. By Proposition~\ref{projbq} and Proposition~\ref{proproj}, the graph $\mathcal{G}(P(\field),d,\ell)$ satisfies the required properties: indeed, $1\in \field$ is an eigenvalue of any element of $\mathsf{SO}(3,\R)\cap\mathsf{GL}(3,\field)$.

Let us build an infinite strongly rotarily transitive connected graph with countably many vertices and diameter $2$. Let $\field$ denote a countable subfield of $\R$ such that every polynomial of degree 3 with coefficients in $\field$ has at least one root in $\field$ and that is closed under taking square roots of positive elements; for instance, take $\field$ to be $\overline{\Q}\cap \R$, where $\overline{\Q}\subset \C$ denotes the algebraic closure of $\Q$. Consider the graph $\mathcal{G}(P(\field),d,\pi/2)$. Since every plane of $\field^3$ that passes through the origin is characterised by its orthogonal $L \subset \field^3$, every automorphism of this graph is a collineation, i.e.\ a bijection $\varphi$ from $P(\field)$ to itself such that for every $p$, $q$ and $r$ in $P(\field)$, if there is a plane of $\field^3$ that contains $p$, $q$ and $r$, then there is one that contains $\varphi(p)$, $\varphi(q)$ and $\varphi(r)$. As a subfield of $\R$ which is closed under taking square roots of positive elements, $\field$ has no non-trivial automorphism: the fundamental theorem of projective geometry thus guarantees that every collineation of $P(\field)$ is induced by an element of $\mathsf{GL}(3,\field)$. We conclude by noticing that, as every polynomial of degree 3 with coefficients in $\field$ has at least one root in $\field$, every element of $\mathsf{GL}(3,\field)$ has an eigenvector in $\field^3\backslash\{0\}$.
\end{proof}

\begin{rem}
In the proof of Theorem~\ref{county}, if one was not looking for countable examples, one could take $\field=\R$, thus providing more visual examples. Also notice that we have established connectedness in a direct way, without resorting to the Axiom of Choice. Finally, it must be noted that every graph with diameter 1 must be complete and have at least 2 vertices, and in particular cannot be strongly rotarily transitive if one assumes the Axiom of Choice.
\end{rem}

\section{Several questions}

\label{sec:quest}

In this section, we collect several questions on rotarily transitive graphs.
Recall that a graph is \defini{nice} if it is non-empty, connected and locally finite.

\begin{quest}
\label{questrong}
Is there a nice infinite strongly rotarily transitive graph?
\end{quest}

\begin{quest}
Is there a nice infinite graph that can be endowed with a free transitive action but can also be endowed with a rotarily transitive action?
\end{quest}

\begin{quest}
Which are the groups that can act in a rotarily transitive and faithful way on a (nice) graph?
\end{quest}

\begin{quest}
Is there an infinite strongly rotarily transitive connected graph with (countably many vertices and) infinite diameter?
\end{quest}

The \defini{degree} of a vertex in a graph is its number of neighbours. The degree of a non-empty transitive graph is the unique value taken by its degree function, i.e.~the degree of any of its vertices.

\begin{quest}
What is the minimal degree of a nice infinite (strongly) rotarily transitive graph?
\end{quest}

We have seen at the end of Section~\ref{algebra} that there is a nice graph that can be endowed with a rotarily transitive group action such that any vertex has a finite stabiliser.

\begin{quest}
Is there a nice graph that can be endowed with a faithful rotarily transitive group action such that any vertex has an infinite stabiliser?
\end{quest}

\begin{quest}
Is there a nice infinite strongly rotarily transitive graph $\mathcal{G}$ such that the stabiliser of any vertex for the usual action $\Aut(\mathcal{G})\acts \mathcal{G}$ is finite?
\end{quest}

\begin{quest}
Is there a nice infinite strongly rotarily transitive graph $\mathcal{G}$ such that the stabiliser of any vertex for the usual action $\Aut(\mathcal{G})\acts \mathcal{G}$ is infinite?
\end{quest}

\paragraph{Acknowledgements.} I would like to thank Laurent Veysseire for a valuable discussion, and Mark Sapir and David Madore for sharing useful mathematical knowledge. I am also grateful to the Weizmann Institute of Science and my postdoctoral hosts, Itai Benjamini and Gady Kozma, for the freedom of research they have provided to me.

\end{document}